\documentclass[11pt,reqno]{amsart}

\usepackage{amsmath,amsfonts,amssymb,latexsym,graphics,epsfig,url}
\usepackage{hyperref}
\usepackage{color}
\usepackage{amsthm}
\usepackage{latexsym}
\usepackage{tikz}
\usetikzlibrary{arrows}

\usepackage[english]{babel}
\usepackage{epsfig}
\usepackage{graphicx}

\newcommand{\executeiffilenewer}[3]{%
\ifnum\pdfstrcmp{\pdffilemoddate{#1}}%
{\pdffilemoddate{#2}}>0%
{\immediate\write18{#3}}\fi%
}

\theoremstyle{plain}
\newtheorem{prop}{Property}[section]
\newtheorem{theorem}[prop]{Theorem}
\newtheorem{proposition}[prop]{Proposition}

\newtheorem{lemma}[prop]{Lemma}

\newtheorem{definition}{Definition}[section]
\newtheorem{example}{Example}[section]

\theoremstyle{remark}

\newcommand{\N}{{\mathbb N}}
\newcommand{\F}{{\mathbb F}}
\newcommand{\x}{{\mathbf x}}
\newcommand{\y}{{\mathbf y}}
\newcommand{\z}{{\mathbf z}}
\newcommand{\e}{{\mathbf e}}
\newcommand{\s}[1]{{\overline{#1}}}

\title{On a problem by Shapozenko on Johnson graphs}
 \author{V\'{i}ctor Diego}
 \address{ Department of Mathematics, Universitat Polit\`ecnica de Catalunya, Barcelona}
 \email{victor.diego@upc.edu}
 \thanks{V. Diego acknowledges support from {\em Ministerio de Ciencia e Innovaci\'on}, Spain, and the {\em European Regional Development Fund} under a FPI grant in the project MTM2011-28800-C02-01}
 \author{Oriol Serra}
  \address{ Department of Mathematics, Universitat Polit\`ecnica de Catalunya and Barcelona Graduate School of Mathematics}
  \email{oriol.serra@upc.edu}
  \thanks{O.\,Serra was supported by the Spanish Ministerio de Econom\'ia y Competitividad under project MTM2014-54745-P. }
  \author{Llu\'{i}s Vena}
  \address{Computer Science Institute of Charles University (IUUK and ITI), Prague}
  \email{lluis.vena@gmail.com}
\thanks{L.\,Vena  was supported by the Center of Excellence-Inst. for Theor. Comp. Sci., Prague, P202/12/G061, and by Project ERCCZ LL1201 \emph{CORES}}

\date{}

\begin{document}
\baselineskip 14pt

\begin{abstract}
	The Johnson graph $J(n,m)$ has the $m$--subsets of $\{1,2,\ldots,n\}$ as vertices and two
	subsets are adjacent in the graph if they share $m-1$ elements. Shapozenko asked
	about the isoperimetric function $\mu_{n,m}(k)$ of Johnson graphs, that is, the cardinality of the smallest boundary of sets with $k$ vertices in $J(n,m)$ for each $1\le k\le {n\choose m}$. We give an upper bound for $\mu_{n,m}(k)$ and show that, for each given $k$ such that the solution to the Shadow Minimization Problem in the Boolean lattice is unique, and each  sufficiently large $n$, the given upper bound is  tight. We also show that the bound is tight for the small values of $k\le m+1$ and for all values of $k$ when $m=2$.
	
	{\bf Keywords:} Johnson graph, Isoperimetric problem, Shift compression.
\end{abstract}

\maketitle

%%%%%%%%%%%%%%%%   NEW VERSION  %%%%%%%%%%%%%%

\section{Introduction}
\label{sec:intr}

Let  $G=(V,E)$ be a graph. Given a set $X\subset V$ of vertices, we denote by
$$
\partial X=\{ y\in V\setminus X: d(X,y)=1\},\; B(X)=\{ y\in V: d(X,y)\le 1\} =X\cup \partial X,
$$
the {\it boundary} and the {\it ball} of $X$ respectively, where $d(X,y)$ denotes $\min\{d(x,y):x\in X\}$.

We write $\partial_G$ and $B_G$ when the reference to $G$ has to be made explicit. The {\it vertex-isoperimetric function} (we will call it simply {\it isoperimetric function}) of $G$ is defined as
$$
\mu_G (k)=\min \{|\partial X|: X\subset V,\; |X|=k\},
$$
that is, $\mu_G (k)$ is the size of the smallest boundary among sets of vertices with cardinality  $k$.

The isoperimetric function is known only for a few classes of graphs. One of the seminal results is the exact determination of the isoperimetric function for the $n$--cube obtained by Harper \cite{harper66} in 1966 (and by Hart with the edge--isoperimetric function at \cite{Hart76} in 1976.) Analogous results were obtained for cartesian products of chains by Bollob\'as and Leader \cite{BL90} and Bezrukov \cite{bezrukov90}, cartesian products of even cycles by Karachanjan \cite{Karachanjan82} and Riordan \cite{Riordan98} (see also Bezrukov and Leck at \cite{bezrukovleck09}) and some other cartesian products by Bezrukov and Serra \cite{bezrukovserra02}. 

The Johnson graph $J(n,m)$ has  the $m$--subsets of $[n]=\{1,2,\ldots,n\}$ as  vertices and two $m$--subsets are adjacent in the graph whenever their symmetric difference has cardinality $2$.  It follows from the definition that, for $m=1$, the Johnson graph $J(n,1)$ is the complete graph $K_n$. For $m=2$ the Johnson graph $J(n,2)$ is the line graph of the complete graph on $n$ vertices, also known as the triangular graph $T(n)$. Thus, for instance, $J(5,2)$ is  the complement of the Petersen graph, displayed in Figure \ref{fig:j52}. Also, $J(n,2)$ is the complement of the Kneser graph $K(n,2)$, the graph which has the $2$--subsets of $[n]$ as vertices and two pairs are adjacent whenever they are disjoint.

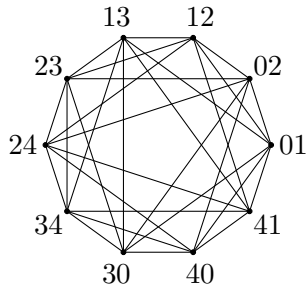
\begin{figure}[h]
	\begin{center}
		\begin{tikzpicture}[scale=0.3]

		\foreach \i in {0,...,9}
		{
			\path (36*\i:5) coordinate (P\i);
			\path (36*\i:6) coordinate (Q\i);
			\draw[fill] (P\i) circle (3pt);
		}
		
		\node  at (Q0) {$01$};
		\node  at (Q1) {$02$};
		\node  at (Q2) {$12$};
		\node  at (Q3) {$13$};
		\node  at (Q4) {$23$};
		\node  at (Q5) {$24$};
		\node  at (Q6) {$34$};
		\node  at (Q7) {$30$};
		\node  at (Q8) {$40$};
		\node  at (Q9) {$41$};
		
		\foreach \i / \j  in {0/1,1/2,2/3,3/4,4/5,5/6,6/7,7/8,8/9,9/0}
		{
			\draw (P\i) --(P\j);
		}
		
		\foreach \i / \j  in {0/2,2/4,4/6,6/8,8/0,0/3,3/6,6/9,9/2,2/5,5/8,8/1,1/4,4/7,7/0}
		{
			\draw (P\i) --(P\j);
		}
		
		\foreach \i / \j  in {1/5,1/7,3/7,3/9, 5/9}
		{
			\draw (P\i) --(P\j);
		}
		
		\end{tikzpicture}
		\caption{The Johnson graph $J(5,2)$.}
		
	\end{center}
\end{figure}\label{fig:j52}

Johnson  graphs arise from the association schemes named after Johnson who introduced them, see e.g. \cite{Cam99}.The Johnson graphs are one of the important classes of distance--transitive graphs; see e.g.  Brouwer, Cohen, Neumaier \cite[Chapter 9]{browercohen1989} or  Godsil  \cite[Chapter 11]{godsil1993}.

Given a family $S$ of $m$--sets of an $n$--set, its \textit{lower shadow} $\Delta (S)$ is the family of $(m-1)$--sets which are contained in some $m$--set in $S$. The \textit{upper shadow} $\nabla (S)$ of $S$ is the family of $(m+1)$-sets which contain some $m$--set in $S$. The \textit{ball} of $S$ in the Johnson graph $J(n,m)$ can be written as
\begin{equation}
\label{deltas}
B(S)=\nabla(\Delta (S))=\Delta (\nabla (S)).
\end{equation}

These equalities establish a connection between the isoperimetric problem in the Johnson graph with the Shadow Minimization Problem (SMP) in the Boolean lattice, which consists in finding, for a given $k$, the smallest cardinality of $\Delta (S)$ among all families $S$ of $m$--sets with cardinality $k$. The latter problem is solved by the well--known Kruskal--Katona theorem \cite{Kruskal63,Katona68}, which establishes that the initial segments in the colex order provide a family of extremal sets for the SMP.

Recall that the {\it colex order} in the set of $m$--subsets of $[n]$ is defined as $X\le Y$ if and only if $\max ((X\setminus Y)\cup(Y\setminus X))\in Y$ (we follow here the terminology from Bollob\'as \cite[Section 5]{bollobas}; we also use ${[n]\choose m}$ to denote the family of $m$--subsets of an $n$--set, and $[k,l]=\{k,k+1,\ldots ,l\}$ for integers $k<l$.) The computation of the boundary of initial segments in the colex order (the family of the first $m$-- subsets in this order) provides the following upper bound for the isoperimetric function of Johnson graphs:

\begin{proposition}\label{prop:upper} Let $\mu_{n,m}: [N]\to \N$ denote the isoperimetric function of the Johnson graph $J(n,m)$, where $N= {n\choose m}$. Let
	$$
	k={k_0\choose m}+{k_1\choose m-1}+\cdots +{k_r\choose m-r},\;  k_0> \cdots > k_r\ge m-r>0,
	$$
	be the $m$--binomial representation of $k$. Then
	\begin{equation}\label{eq:ub}
	\mu_{n,m}(k)\le f(k,n,m),
	\end{equation}
	where
	\begin{equation}\label{eq:ub1}
	f(k,n,m)=
	\binom{k_0}{m-1}(n-k_0)+\sum_{i=1}^{r}\left(\binom{k_i}{m-i-1}(n-k_0-1)-\binom{k_i}{m-i}\right).
	\end{equation}
\end{proposition}

\begin{proof} The initial segment $I$ of length $k$ in the colex order is the disjoint union
	$$
	I=I_0\cup\cdots \cup I_r,
	$$
	where $I_0$ consists of all $m$--sets in ${[k_0]\choose m}$ and, for $j>0$, $I_j$ consists of all sets containing $\{k_{j-1}+1,\ldots, k_0+1\}$ and $m-j$ elements in $[k_j]$. The right hand side of \eqref{eq:ub1} is the cardinality of $\partial I$ as can be shown by induction on $r$. If $r=0$ then $\partial I$ consists of the ${k_0\choose m-1}(n-k_0)$ sets obtained by replacing one element in $[k_0]$ by one element in $[k_0+1,n]$ from a set in $I$. Suppose that $r>0$ and write $I=I'\cup I_r$ as the disjoint union of $I'=I_0\cup \cdots \cup I_{r-1}$ and $I_r$. We have $\partial I=(\partial I'\setminus I_r)\cup (\partial I_r\setminus B(I'))$, the union being disjoint. Since $I_r\subset \partial I'$ we have,
	$$
	|\partial I'\setminus I_r|=|\partial I'|-|I_r|=|\partial I'|-{k_r\choose m-r},
	$$
	while
	$$
	|\partial I_r\setminus B(I')|={k_r\choose m-r-1}(n-k_0-1),
	$$
	since the only sets in $\partial I_r\setminus B(I')$ are those obtained from a set in $I_r$ by replacing one element in $[k_r]$ by one element in $[k_0+2,n]$. %\qed
\end{proof}

The family of initial segments in the colex order  does not provide in general a solution to the isoperimetric problem in $J(n,m)$. A simple example is as follows. 

\begin{example}
	\label{example1}
	Take $n=3(m+1)/2$. The ball $B(\{\x\})$ of radius one in $J(3(m+1)/2,m)$ has cardinality
	$$
	|B_1|=1+m(n-m)=\frac{(m+2)(m+1)}{2}={m+2\choose m},
	$$
	and its boundary has cardinality
	$$
	|\partial B_1|={m\choose 2}{n-m\choose 2}=\frac{m(m-1)(m+3)(m+1)}{16}.
	$$
	On the other hand, according to \eqref{eq:ub1} and the $m$--binomial decomposition of $|B_1|$, the initial segment $I$ of length $|B_1|$ has cardinality
	\begin{align*}
	|\partial I|&={m+2\choose m-1}\frac{m-1}{2}=\frac{(m+2)(m+1)m(m-1)}{12}\\&=|\partial B_1|+\frac{(m+1)m(m-1)^2}{48},
	\end{align*}
	which shows that the unit ball can have, as a function of $m$, an  arbitrarily smaller boundary than the initial segment  in the colex order.%\qed
\end{example}

In his monograph on discrete isoperimetric problems  Leader \cite{leader91}  mentions   the isoperimetric problem for Johnson graphs as one of the intriguing open problems in the area. Later on, in his extensive monograph on isoperimetric problems, Harper \cite{Harper04} atributes the problem to Shapozenko, and recalls that it is still open. Recently, Christofides, Ellis and Keevash \cite{keevash} have obtained a lower bound for the isoperimetric function of Johnson graphs  which is asymptotically tight for sets with cardinality $\frac{1}{2}{n\choose m}$. The Johnson graphs $J(n,2)$ provide a counterexample to a conjecture of Brouwer on the $2$--restricted connectivity of strongly regular graphs, see Cioab\^a, Kim and Koolen  \cite{cioaba2012} and  Cioab\^a,  Koolen  and Li\cite{cioaba2014}, where the connectivity of the more general class of strongly regular graphs and distance--regular graphs is studied. It is also worth mentioning that the edge version of the isoperimetric problem, where the minimization is for the number of edges leaving a set of given cardinality, has also been studied, see e.g.  Ahlswede and Katona \cite{ahslwedekatona78} or Bey \cite{bey2006}. We will only deal with the vertex isoperimetric problem in this paper and refer to it simply as the isoperimetric problem.

Our main purpose in this paper is to show that the initial segments in the colex order still provide a solution to the isoperimetric problem in $J(n,m)$  for many small values of $k$, thus providing the exact value of the isoperimetric function in these cases.

We call a set $S$ of vertices of $J(n,m)$ {\it optimal} if $|\partial (S)|=\mu_{n,m}(|S|)$.   Our first result shows that initial segments in the colex order are optimal sets  in $J(n,2)$.

\begin{theorem}\label{thm:m=2} For each $n\ge 3$ and each $1\le k\le {n\choose 2}$ we have 
	$$
	\mu_{n,2}(k)=f(k,n,2).
	$$
	In particular, the initial segments in the colex order are optimal sets of $J(n,2)$ for each $n\ge 3$.
\end{theorem}

The following theorem allows one to show that  the inequality \eqref{eq:ub} is also tight in $J(n,m)$ for very small sets.

\begin{theorem}\label{thm:m+1} For $k< m-1$ and $n\ge 2(m-1)$ the initial segment of length $k$ of the colex order in $J(n,m)$ is an optimal set. 
\end{theorem}

Our last result, Theorem \ref{thm:asymptotic}, extends Theorem \ref{thm:m+1} in an asymptotic way, by showing that the inequality \eqref{eq:ub} is tight for  a large number of small cardinalities and gives a lower bound for all small cardinalities.

\begin{theorem}\label{thm:asymptotic} Let $k, m$ be positive integers and let 
	$$
	k={k_0\choose m}+{k_1\choose m-1}+\cdots +{k_r\choose m-r},\;  k_0> \cdots > k_r\ge m-r>0,
	$$
	be the $m$--binomial representation of $k$.

	There is $n(k,m)$ such that, for all $n\ge n(k,m)$, the following holds.
	\begin{enumerate}
		\item[{\rm (i)}] If $r<m-1$ then 
		$$
		\mu_{n,m}(k)=f(k,n,m),
		$$
		and  the initial segment in the colex order with length $k$ is the only (up to automorphisms) optimal set with cardinality $k$ of the Johnson graph $J(n,m)$.
		\item[{\rm (ii)}] 
		If $r=m-1$ then
		$$
		\mu_{n,m}(k)\le f(k-k_r,n,m)+k_r.
		$$
	\end{enumerate}
\end{theorem}

The proof of Theorem \ref{thm:asymptotic} provides the  estimation 
$$
n(k,m)\le m+k+1-\mu_{m+k+1,m}(k)+f(k,m+k+1,m)
$$
for the value of $n(k,m)$ above for which the statement of Theorem \ref{thm:asymptotic} holds. This upper bound for $n(k,m)$ is not tight but we  make no attempt to optimize its value in this paper.  

Example \ref{example1} shows that the initial segment in the colex order of length ${m+2\choose m}$ can fail to be an optimal set in $J(n,m)$ if  $n=3(m+1)/2$. In the last section  we describe another infinite family of examples for which the initial segment in the colex order  fails again to be  an optimal set in $J(n,m)$ for every fixed $m$  and all $n$ large enough. 

\begin{proposition}\label{prop:nocolex} Let $m$ be a positive integer. For each  integer $k$ of the form
	$$
	k={t\choose m}+3{t\choose m-1}
	$$
	with $t$ sufficiently large with respect to $m$ 
	there is a set $S$ with cardinality $k$  such that 
	$$
	|B(S)|<f(k,n,m)
	$$
	for all $n\geq t+3$.
\end{proposition} 

When $n=t+3$ the set $S$ in Proposition \ref{prop:nocolex} can be easily described as the ball in $J(n,m)$ of ${[t]\choose m}$, the family of $m$--subsets of the first $t$ symbols. Such sets are clear candidates to be optimal sets. The examples described in Proposition \ref{prop:nocolex} are closely related to the non--unicity of solutions to the Shadow Minimization Problem in the Boolean lattice (see Theorem \ref{thm:fg} below.)

Standard compression techniques are used to prove the above results. These tools fall short to solve the isoperimetric problem of Johnson graphs in full mainly because, as pointed out in \cite{keevash}, for instance, optimal sets in Johnson graphs do not have the nested property (the ball of an optimal set is not optimal.) However these techniques are still useful to show that the colex order provides a sequence of extremal sets for small cardinalities.

The paper is organized as follows. Section \ref{sec:comp} recalls the shifting techniques and compression of sets. The proofs of theorems \ref{thm:m=2}, \ref{thm:m+1} and \ref{thm:asymptotic} are given in sections \ref{sec:m2}, \ref{sec:m+2} and \ref{sec:nlarge} respectively. In the proof of Theorem \ref{thm:asymptotic} we use a result by F\"uredi and Griggs \cite{furedigriggs} which characterizes the cardinalities for which the Shadow Minimization Problem for the Boolean lattice has unique solution. The statement below is a rewriting of a combination of Proposition 2.3 and Theorem 2.6 in \cite{furedigriggs}.

\begin{theorem}[\cite{furedigriggs}]\label{thm:fg} Let
	$$
	k={k_0\choose m}+{k_1\choose m-1}+\cdots +{k_r\choose m-r},\;  k_0> \cdots > k_r\ge m-r>0,
	$$
	be the $m$--binomial representation of $k$. 
	
	The initial segment in the colex order is the unique (up to automorphisms) solution to the Shadow Minimization Problem in the Boolean lattice  if and only if $r<m-1$.
\end{theorem}  

Finally in Section \ref{sec:nonoptimal} we prove Proposition \ref{prop:nocolex}. The result  describes an infinite family of examples which show that the initial segments in colex order may fail to be optimal sets. The nature of this example shows that the isoperimetric problem in Johnson graphs still has many intriguing open questions to be solved.

\section{Shifting techniques}\label{sec:comp}

Shifting techniques are one of the key tools in the study of set systems. They were initially introduced in the original proof of the Erd\H os--Ko--Rado theorem \cite{ekr} and have been particularly used in the solution by Frankl and F\"uredi \cite{franklfuredi81}  of the isoperimetric problem for hypercubes.

In what follows we identify subsets of $[n]$ with their characteristic vectors   $\x=(x_1,\ldots ,x_n)\in \{0,1\}^n$ where  $x_i=1$ if $i$ is in the corresponding set and $x_i=0$ otherwise. We denote the {\it support} of $\x$ by
$$
\s{\x}=\{ i: x_i=1\}.
$$
and the $\ell_1$-{\it norm} of $\x$ by
$$
|\x|=\sum_i x_i.
$$
The support $\s{S}$ of   $S\subset \{0,1\}^n$ is the union of the supports of its vectors. We often identify  a set $S\subset \{0,1\}^n$ with the subset in $\{0,1\}^{n'}$, $n'>n$, obtained  by adding zeros to the right in the coordinates of its vectors. Thus, the initial segment  of length $k$ is considered to be a subset of $\{0,1\}^n$ for each sufficiently large $n$. 

The sum
$
\x+ \y=(x_1+y_1 \pmod 2,\ldots ,x_n+y_n\pmod 2),
$
of characteristic vectors is meant to be performed in the field $\F_2^n$ and it corresponds to the symmetric difference of the corresponding sets.  We also denote by $\e_1,\ldots ,\e_n$ the unit vectors with $1$ in the $i$--th coordinate and zero everywhere else.

With the above notation, the set of vertices of the Johnson graph $J(n,m)$ are all vectors of $\{0,1\}^n$ with norm $m$, and the neighbors of $\x$ in $J(n,m)$ are the vectors
$$
\x+\e_i+ \e_j,
$$
for each pair $i, j$ such that $x_i+x_j=1$.

We next recall the definition of the shifting transformation.

\begin{definition} Let $i,j\in [n]$. For a set $S\subset \{0,1\}^n$ define
	$$
	S_{ij}=\{ \x\in S: x_i=1 \;\mbox{ and } \; x_j=0\},
	$$
	and
	$$
	T_{ij}(\x,S)=\left\{\begin{array}{ll} \x+\e_i+\e_j, &\mbox{if}\;  \x\in S_{ij} \; \mbox{and }\; \x+\e_i+\e_j\not\in S\\
	\x&\mbox{otherwise}\end{array}\right.
	$$
	The $ij$--shift of $S$ is defined as
	$$
	T_{ij}(S)=\{ T_{ij}(\x,S):\x\in S\}.
	$$
\end{definition}

It follows from the definition that the shifting  $T_{ij}$ of a set preserves its cardinality and the norm of its elements. Moreover,  it sends every vertex to a vertex at distance at most $1$.  The main property of the shifting transformation is that it does not increase the cardinality of the ball of a set. This property follows from the analogous ones for upper and lower shadows. We include a direct proof here for completeness.

\begin{lemma}\label{lem:contrac}
	Let $i,j\in [n]$  and write $T=T_{ij}$. For each set $S$ of vertices in the Johnson graph $J(n,m)$ we have
	\begin{equation}\label{eq:contrac2}	
	B(T(S))\subseteq T(B(S)).
	\end{equation}
	In particular,
	\begin{equation}\label{eq:contrac3}	
	|B(T(S))|\leq |T(B(S))|=|B(S)|.
	\end{equation}
\end{lemma}

\begin{proof} We will show that,
	\begin{equation}\label{eq:contract4}
	\mbox{ for each } \; \y\in T(S), \; \mbox{  we have }  B(\y)\subseteq T(B(S)),
	\end{equation}
	which is equivalent to  (\ref{eq:contrac2}). We observe that then (\ref{eq:contrac3}) follows since $|B(X)|=|X|+|\partial X|$ for every subset $X$ and 
	$$
	|\partial (T(S))|=|B(T(S))|-|T(S)|\le |T(B(S))|-|T(S)|=|B(S)|-|S|=|\partial S|.
	$$
	
	Let $\x$ be the element in $S$ such that $\y=T(\x,S)$. We consider two cases.
	
	{\it Case 1.} $(x_i,x_j)\neq (1,0)$. In this case we certainly have $\y=\x$. Moreover, if $\z\in \partial\x$ such that $(z_i,z_j)=(1,0)$ then it is readily checked that $\z+\e_i+\e_j\in B(\x)$. Therefore the transformation $T(\cdot,B(S))$ leaves  $B(\x)$ invariant. Hence, $B(\x)\subseteq T(B(S)).$
	
	{\it Case 2.} $(x_i,x_j)= (1,0)$.  Then $\z=\x+\e_i+\e_j$  is the only neighbour of $\x$ with $(z_i,z_j)=(0,1)$.

	{\it Case 2.1} $\y=\x$.  Then, by the definition of $T(\cdot,S)$, we have  $\z\in S$ and $T(\z,S)=\z$. Observe that every neighbour $\z'$ of $\x$ is  left invariant by $T(\cdot,B(S))$.
	This is clearly the case if   $(z'_i,z'_j)\neq (1,0)$ and, if $(z_i,z_j)= (1,0)$, because  we then have $\z''=\z'+ \e_i+\e_j\in B(\z)\subset B(S)$. Hence
	$$B(\y)=B(\x)=\cup_{\z\in B(\x)} T( \z, B(S))\subseteq T(B(S)).$$
	
	{\it Case 2.2} $\y\neq \x$. Then $\y\not\in S$ but $\y\in B(\x)\subseteq B(S)$. Each neighbour $\z$ of $\y$ distinct from $\x$ is of the form $\z=\z'+ \e_i+ \e_j$ for some neighbour $\z'$ of $\x$ and therefore it belongs to $T(B(S))$. For $\x$ itself we have $T(\x,B(S))=\x$ because $\y=\x+\e_i+\e_j\in B(S)$. Thus we again have $B(\y)\subset T(B(S))$. This completes the proof of \eqref{eq:contract4}.%\qed
\end{proof}

The weight of a vector $\x \in \{0,1\}^n$ is
$$
w(\x)=\sum_{i=1}^n ix_i,
$$
and the weight of  a set $S$ is
$$
w(S)=\sum_{\x\in S} w(\x).
$$
We note that, if $i>j$ then $w(T_{ij}(S))\le w(S)$. Moreover, equality holds if and only if $T_{ij}(S)=S$. Thus,  successive application of transformations $T_{ij}$ using pairs  $i,j$  with $i>j$ eventually produces a set which is stable by any of such transformations. This fact leads to the following definition.

\begin{definition}\label{def:compresed}  We say that a set $S$ is \textit{compressed} if $T_{ij}(S)=S$ for each pair $i,j\in [n]$ with $i>j$.
\end{definition}

Every  set can be compressed by keeping its cardinality and without increasing its boundary. Therefore, in what follows  we can restrict our attention to compressed sets in our study of optimal sets.

\section{The case $m=2$} 
\label{sec:m2}

Theorem \ref{thm:m=2} follows from the following proposition which characterizes compressed optimal sets in $J(n,2)$.

\begin{proposition} A set $S$ of vertices of the graph $J(n,2)$ with cardinality $\binom{t-1}{2}<|S|\leq \binom{t}{2}$  is optimal if and only if, up to isomorphism,  $S\subseteq\binom{[t]}{2}$.
\end{proposition}

\begin{proof} Write $V(J(n,2))$ as the disjoint union
	$$
	V(J(n,2))=S\cup \partial S\cup \tilde{S},
	$$
	where $\tilde{S}$ is the set of vertices at distance two from  $S$. Let $t=\left|\s{S}\right|$ be the amount of elements of $[n]$ in the support $\s{S}$ of $S$. The only vectors in $\tilde{S}$ are the ones which have both nonzero coordinates in $[n]\setminus \s{S}$. Therefore
	$$
	|\tilde{S}|={n-t\choose 2},
	$$
	and
	$$
	|\partial S|={n\choose 2}-|S|-{n-t\choose 2}.
	$$
	Hence, for a given cardinality $|S|$, $|\partial S|$ is a increasing function of $t$ alone. The optimal value is therefore obtained when $t$ is smallest possible, which is the smallest $t$ such that $|S|\le {t \choose 2}$. This is achieved for any subset $S\subset {[t]\choose 2}$ if $|S|>{t-1\choose 2}$.%%\qed
\end{proof}

As a consequence of the above proposition, we can see that the solution to the isoperimetric problem in $J(n,2)$  is unique (up to isomorphism) for sets of cardinality ${t\choose 2}$ (and also for sets of cardinality $\binom{t}{2}-1$).

\section{Small sets}
\label{sec:m+2}

In this section we prove Theorem \ref{thm:m+1}.
Consider the partition
$$
V(J(n,m))=\{\mathbf{x}\in V(J(n,m)) |  x_n=0\} \cup \{\mathbf{x}\in V(J(n,m)) |  x_n=1\}=V_0\cup V_1.
$$
The subgraph of $J(n,m)$ induced by $V_0$ is isomorphic to $J(n-1,m)$ and the subgraph induced by $V_1$is isomorphic to $J(n-1,m-1)$. There is an edge in $J(n,m)$ joining $\x\in V_0$ with $\y\in V_1$ if and only if $\s{\y}\setminus\{n\}\subset \s{\x}$.

\begin{lemma}\label{lem:section} Let  $S$ be a   set of vertices  in $J(n,m)$. Let $S_0=S\cap \{ x_n=0\}$ and $S_1=S\cap \{ x_n=1\}$. If $S$ is compressed then
	$$
	B(S)=B(S_0)\; \mbox{ and } |B(S)|= |B'(S'_0)|+ |\Delta (S'_0)|,
	$$
	where $S'_0=\{ \x\in \{0,1\}^{n-1}: (\x, 0)\in S_0\}$ and $B'$ denotes the ball in $J(n-1,m)$.
\end{lemma}

\begin{proof}  Let $\x \in S_1$ and $i\not\in  \s{\x}$. Since $S$ is compressed, we have
	$$
	\y=\x+ \e_n + \e_i\in S_0,
	$$
	which implies $\x=\y+ \e_n+ \e_i\in B(S_0)$. Hence, $S_1\subset B(S_0)$. Moreover, if $j\in \s{\x}$ then
	$$
	\x+ \e_i + \e_j=(\x+ \e_n + \e_i) + \e_n + \e_j\in B(S_0),
	$$
	so that $B(\x)\subset B(S_0)$.
	Hence $B(S)=B(S_0)\cup B(S_1)= B(S_0)$. This proves the first part of the sentence.
	
	For the second part, we just note that $B(S_0)$ is the disjoint union $(B(S_0)\cap V_0)\cup (B(S_0)\cap V_1)$. Since the subgraph induced by $V_0$ is isomorphic to $J(n-1,m)$, we have  $|B(S_0)\cap V_0|=|B'(S'_0)|$. On the other hand,  there is an edge in $J(n,m)$ joining $x\in V_1$ with $y\in V_0$ if and only if $\s{x}\setminus\{n\}\subset \s{y}$. It follows that $B(S_0)\cap V_1=\Delta (S'_0)$.%%\qed
\end{proof}

The next proposition considers the case of Johnson graphs $J(n,m)$ when $n=2m-2$. 

\begin{proposition}\label{prop:2m-2}
	For each $k=1,\ldots ,{2m-2 \choose m}$ the initial segment of length $k$ in   the colex order  is an optimal set of the graph  $J(2m-2,m)$.
\end{proposition}

\begin{proof}
	Let us  recall \eqref{deltas} and express the ball of a set $S$ as
	$$B(S)=\nabla(\Delta (S)).$$
	By the Kruskal--Katona theorem, $\Delta S$ is minimized by the initial segment $I_k$, $k=|S|$,   in the colex order of ${[2m-2]\choose m}$. Moreover, $\Delta I_k$ is also an initial segment of the colex order in ${[2m-2\choose m-1}$. 
	
	The automorphism of the Boolean lattice given by taking complements sends  the middle level ${[2m-2\choose m-1}$ to itself, and the initial segments in the colex order are exchanged by the initial segments in the lexicograhic order. It follows that the initial segments of the colex order are also a solution to the minimization of the {\it upper} shadow (as well as the initial segments in the lexicographic order). Hence, $I_k$ minimimizes $|\nabla(\Delta (S))|$. %%\qed
\end{proof}

We use the Proposition \ref{prop:2m-2} as the base case of the induction for  the proof of Theorem \ref{thm:m+1}.

\begin{proof}[of Theorem \ref{thm:m+1}.] The proof is by induction on $n$. For $n=2m-2$ the result follows from Proposition \ref{prop:2m-2}. Let $S$ be a compressed set of cardinality $k\le m-1$ in $J(n,m)$, $n\ge 2m-1$ and consider its decomposition $S=S_0\cup S_1$. Since $S$ is compressed, every element in  $S_1$ gives rise to at least $m-1$ elements in $S$. Since  $k<m-1$ we have $S_1=\emptyset$. 
	
	By Lemma \ref{lem:section}, the cardinality of the ball of $S$ is $$|B(S)|= |B'(S'_0)|+ |\Delta (S'_0)|. $$ By the induction hypothesis, the initial segment in the colex order minimizes the ball $|B'(S'_0)|$ in $J(n-1,m)$ as well as, by teh Kruskal--Katona theorem, the lower shadow  $|\Delta (S'_0)|$. %%\qed
\end{proof}

\section{Optimal sets for large $n$}\label{sec:nlarge}

In this Section we give the proof of Theorem \ref{thm:asymptotic}. In what follows we call a positive integer $k$ {\it critical} if its $m$--binomial representation has length $m$ (namely, it has $r=m-1$).

\begin{proof}[of Theorem \ref{thm:asymptotic}.]  Let $S$ be an optimal set with cardinality $k$ in $J(n,m)$. We may assume that $S$ is compressed. Let $n_0$ be such that the support of every element in $S$ is contained in $[n_0]$. Since $S$ is compressed, if the support of $\x\in S$ contains $n_0$ then  we have $T_{n_0i}(\x)\in S$ for each $i\in [n_0-1]\setminus \s{\x}$. It follows that $n_0\le m+k+1$.  For each $n\ge n_0$ every element in $\Delta (S)$ gives rise to $n-n_0$ distinct vectors in $\partial S$ which have a coordinate in $[n_0+1,n]$ and therefore are disjoint from the ball $B_0(S)$ in $J(n_0,m)$. Moreover every two such vectors which only differ in their coordinate from $[n_0+1,n]$ come from a unique element in $\Delta(S)$. Therefore, we have
	$$
	|B(S)|=|B_0(S)|+(n-n_0)|\Delta S|,
	$$
	where $B_0$ denotes the ball of $S$ in $J(n_0,m)$ and $B$ denotes the ball of $S$ in $J(n,m)$. Similarly, for $I$ be the initial segment of length $k$ in the colex order. We have
	$$
	|B(I)|=|B_0(I)|+(n-n_0)|\Delta I|
	$$
	Hence,
	\begin{equation}\label{eq:compare}
	|B(S)|=|B(I)|+(|B_0(S)|-|B_0(I)|)+(n-n_0)(|\Delta (S)|-|\Delta (I)|).
	\end{equation}
	If $|\Delta (S)|>|\Delta (I)|$ then we have $|B(S)|>|B(I)|$ for each sufficiently large $n$. By  Theorem \ref{thm:fg}, if the $m$--binomial representation of $k$ has less than $m$ terms then the initial segment in the colex order is the unique solution to the Shadow Minimization Problem. It follows that, if $S\neq I$ then $|B(S)|>|B(I)|$ for all 
	$n>n_0-|B_0(S)|+|B_0(I)|$. This proves the first part of Theorem \ref{thm:asymptotic} and gives the estimate 
	$$
	n(k,m)\le m+k+1-\mu_{m+k+1,m}(k)+f(k,m+k+1,m).
	$$
	
	On the other hand, we have $\mu_{n,m}(k)\le \mu_{n,m}(k+1)-1$, since otherwise an optimal set $X$ with cardinality $k+1$ satisfies $|\partial (X\setminus \{x\})|<\mu_{n,m} (k)$ for every $x\in X$, a contradiction with the definition of $\mu_{n,m}$. Suppose that $k$ is a critical integer and let $k,k-1,\ldots ,k-\ell+1$ be the longest decreasing sequence of critical integers. By the above remark we have  $\mu_{n,m} (k)\le \mu_{n,m}(k-\ell)+\ell$ and $\mu_{n,m} (k-\ell)=f(k-\ell,n,m)$, the cardinality of the boundary of an initial segment in colex order with length $k-\ell$. The value of $\ell$ is clearly 
$k_r$. This proves the second part of the statement. %\qed
\end{proof}

% We recall that the above proof gives the estimation 
% $$
% n(k,m)\le m+k+1+\mu_{m+k+1,m}(k)-f(k,m+k+1,m)
% $$
% for the minimum value of $n$ from which the statement of  Theorem \ref{thm:asymptotic} holds.

\section{Initial segments which are not optimal}\label{sec:nonoptimal}

We conclude the paper by proving Proposition \ref{prop:nocolex}, which shows that there are values of $k$  for which  the initial segment of length $k$ in the colex order fails to be an optimal set of $J(n,m)$ for all sufficiently large $n$.

We prove first that, for each $m$ and each integer $t$ sufficently large with respect to $m$, 
$$
g(t,m)={t\choose m}+3{t\choose m-1}
$$
is a critical cardinality, namely, the $m$-binomial expansion of $g(t,m)$  has $m$ terms. This means that the solution of the Minimal Shadow Problem  is not unique for $k=g(t,m)$. This fact is  used in the proof of Proposition \ref{prop:nocolex}. 

\begin{lemma}\label{lem:gtm} There is an infinite strictly increasing integer sequence $\{\lambda_i\}_{i\geq 0}$, $\lambda_i+1<\lambda_{i+1}$ such that, for each $t$ and each $m\ge 1$,
	\begin{equation}\label{eq:gtm}
	g(t,m)=\sum_{i=0}^{m-1} {t-\lambda_i\choose m-i}+1.
	\end{equation}
\end{lemma}

\begin{proof} By induction on $m$. For $m=1$ we have 
	\begin{equation}\label{eq.first_ind_step}
	g(t,1)=t+3={t+2\choose 1}+1,
	\end{equation}
	and for $m=2$,
	\begin{equation}\label{eq.second_ind_step}
	g(t,2)={t+2\choose 2}+{t-2\choose 1}+1,
	\end{equation}
	giving $\lambda_0=-2$ and $\lambda_1=2$.
	By using ${n\choose m}=\sum_{j=0}^{n-1}{j\choose m-1}$ and induction, for   $m\ge 3$ we have
	\begin{align*}
	g(t,m)=&{t\choose m}+3{t\choose m-1}\\
	=&\sum_{j=0}^{t-1}{j\choose m-1}+3\sum_{j=0}^{t-1}{j\choose m-2}\\
	=&\sum_{j=0}^{t-1}g(j,m-1)\\
	=&\sum_{j=0}^{t-1} \left(\sum_{i=0}^{m-2}{j-\lambda_i\choose m-1-i}+1\right)\\
	=&\sum_{i=0}^{m-2}\sum_{j=0}^{t-1}{j-\lambda_i\choose m-1-i}+t.
	\end{align*}
	By using $\lambda_0=-2$ and $\lambda_i>0$ for $i\in[1,m-2]$, we can write
	\begin{align*}
	g(t,m)=&\sum_{j=0}^{t-1}{j-\lambda_0\choose m-1}+\sum_{i=1}^{m-2}\left(\sum_{j=\lambda_i}^{t-1}{j-\lambda_i\choose m-1-i}+\sum_{j=0}^{\lambda_i-1}{j-\lambda_i\choose m-1-i}\right)+t\\
	&={t-\lambda_0\choose m}+\sum_{i=1}^{m-2}{t-\lambda_i\choose m-i}+t+\sum_{i=1}^{m-2}\sum_{j=0}^{\lambda_i-1}{j-\lambda_i\choose m-1-i},
	\end{align*}
	which shows that \eqref{eq:gtm} holds with
	\begin{align}
	\lambda_{m-1}=&-\sum_{i=1}^{m-2}\sum_{j=0}^{\lambda_i-1}{j-\lambda_i\choose m-1-i}+1\nonumber \\
	=&-\sum_{i=1}^{m-2}\sum_{\ell=1}^{\lambda_i}{-\ell\choose m-i-1}+1 \nonumber \\
	=&-\sum_{i=1}^{m-2}\sum_{\ell=1}^{\lambda_i}(-1)^{m-i-1}{m-i-2+\ell\choose m-i-1}+1\nonumber \\
	=&\sum_{i=1}^{m-2}(-1)^{m-i}{m-i-1+\lambda_i\choose m-i}+1\nonumber \\
	=&\sum_{j=2}^{m-1}(-1)^{j}{\lambda_{m-j}+j-1\choose j}+1 \label{eq.l_def}
	\end{align}
	We observe that the sequence is uniquely determined once $\lambda_1$ is fixed. 
	The first values of the sequence are
	$$
	-2,2, 4,7,14,51,928, 409625, \ldots
	$$
	It remains to show that the sequence is increasing. We will in fact show that $\lambda_m\ge \max\{ \lambda_{m-1}+2,\lambda_{m-1}^2/4\}$ for all $m\ge 2$.
	The above inequality holds for $m\le 7$ as shown by the first values of the sequence. By using \eqref{eq.l_def} (with $\lambda_m$ instead of $\lambda_{m-1}$), we have, 
	$$
	\lambda_{m}\ge \sum_{j=2, j \text{ even}}^{2\lfloor m/2\rfloor -1}\left({\lambda_{m-j+1}+j-1\choose j}-{\lambda_{m-j}+j\choose j+1}\right).
	$$
	For $j=2$ we have
	\begin{align*}
	{\lambda_{m-1}+1\choose 2}-{\lambda_{m-2}+2\choose 3}&=\frac{\lambda_{m-1}^2}{4}+\frac{\lambda_{m-1}(\lambda_{m-1}+2)}{4}-{\lambda_{m-2}+2\choose 3}\\&\geq \frac{\lambda_{m-1}^2}{4}+\frac{\lambda_{m-2}^4}{64}+\frac{2\lambda_{m-2}^2}{16}-{\lambda_{m-2}+2\choose 3}\\
	&>\frac{\lambda_{m-1}^2}{4}
	\end{align*}
	where the last inequality holds  (for $m\ge 4$) because the largest root of the polynomial $x^4/64+x^2/8-{x+2\choose 3}$ is smaller than $\lambda_4=14$.

	On the other hand, for $j\ge 4$, 
	\begin{align*}
	{\lambda_{m-j+1}+j-1\choose j}&-{\lambda_{m-j}+j\choose j+1}=\frac{1}{j!}\left(\prod_{t=0}^{j-1} (\lambda_{m-j+1}+t)-\frac{\prod_{t=0}^j(\lambda_{m-j}+t)}{j+1}\right)\\
	&\stackrel{ \lambda_{m-j+1}> \lambda_{m-j}+1}{>}\frac{1}{j!}\prod_{t=2}^j(\lambda_{m-j}+t)\left(\lambda_{m-j+1}-\frac{\lambda_{m-j}(\lambda_{m-j}+1)}{j+1}\right)\\
	&\stackrel{ \lambda_{m-j+1}\ge \lambda_{m-j}^2/4}{\ge}\frac{1}{j!}\prod_{t=2}^j(\lambda_{m-j}+t)\left(\frac{\lambda_{m-j}^2}{4}-\frac{\lambda_{m-j}(\lambda_{m-j}+1)}{j+1}\right).
	\end{align*}
	The right-hand side is nonnegative if $m-j\ge 2$, as then  $\lambda_{m-j}\ge 4$. If $m-j=1$ then it follows by induction on $j\ge 1$ that
	$$
	{\lambda_2+j-1\choose j}-{\lambda_1+j\choose j+1}={3+j\choose j}-{2+j\choose j+1}>0.
	$$
	This completes the proof. %\qed
\end{proof}

For $t$ larger than $\lambda_{m-1}$ the equality \eqref{eq:gtm} in Lemma \ref{lem:gtm} provides the $m$--binomial expansion of $g(t,m)$. Hence this binomial  expansion has length $m$ and, by Theorem \ref{thm:fg}, $g(t,m)$ is a critical cardinality.  
The proof of  Proposition \ref{prop:nocolex} uses this fact by choosing two distinct optimal sets for the SMP problem which have different boundaries in the Johnson graph.

\begin{proof} {\it of Proposition \ref{prop:nocolex}}
	Let  $S=B_0({[t]\choose m})$, the ball of ${[t]\choose m}$ in the Johnson graph $J(n_0,m)$, $n_0=t+3$. The cardinality of $S$ is  
	$$
	k={t \choose m}+3{t\choose m-1}=g(t,m).
	$$
	Let $I(k)$ denote the initial segment of length $k$ in the colex order. 
	
	By Lemma \ref{lem:gtm} the $m$-binomial expansion of $k$ has $m$ terms and can be written as
	\begin{equation}
	{t \choose m}+3{t\choose m-1}={t-\lambda_0 \choose m} + {t-\lambda_1 \choose m-1} + \ldots + {t-\lambda_{m-1} \choose 1}.
	\label{expansion}
	\end{equation}
	We note that the shadows of $S$ and of $I(k)$ have the same cardinality:
	\begin{align*}
	|\Delta S|&={t \choose m-1}+3{t\choose m-2}\\&={t-\lambda_0 \choose m-1} + {t-\lambda_1 \choose m-2} + \ldots +{t-\lambda_{m-2}\choose 1}+ {t-\lambda_{m-1} \choose 0}\\&=|\Delta (I(k))|.
	\end{align*}
	
	It can be readily checked that the boundary of $S$ in $J(n_0,m)$ has cardinality
	$$
	|\partial S|=\left|\partial^2 {[t] \choose m}\right|=3{t \choose m-2},
	$$
	On the other hand, the boundary in $J(n_0,m)$ of the initial interval $I(k)$ as given by the function $f(k,n_0,m)$ in \eqref{eq:ub1} is
	\begin{align*}
	|\partial I(k)|=&{t-\lambda_0 \choose m-1}+\sum_{i=1}^{m-1} \left({t-\lambda_i\choose m-i-1}(n-t+\lambda_0-1)-{t-\lambda_i \choose m-i}\right)\\
	=&{t+2 \choose m-1}-\sum_{i=1}^{m-1} {a_i \choose m-i}={t+2 \choose m-1}+{t+2\choose m} -{t\choose m}-3{t\choose m-1}\\
	=&{t+1\choose m-2}+2{t\choose m-2},
	\end{align*}
	which is strictly larger than $3{t\choose m-2}$. Thus $|B_0(I(k))|>|B_0(S)|$. Moreover, by \eqref{eq:compare}, for all $n\ge n_0$, we have
	$$
	|B(I(k))|=|B(S)|+(|B_0(I(k))|-|B_0(S)|)>|B(I(k)|.
	$$
	Therefore the intial segment in colex order $I$ fails to be an optimal set for all $n\ge t+3$.   %\qed
\end{proof}

\section*{Acknowledgements}
 The authors are grateful to the comments and remarks of the referees, pointing out some inaccuracies in the original manuscript and helping to improve the readability of the paper.

%Previously
\bibliographystyle{abbrv}	
\bibliography{to_arxiv.bib}

\begin{thebibliography}{10}

\bibitem{browercohen1989}
A.~C. A.E.~Brouwer and A.~Neumaier.
\newblock {\em Distance-Regular Graphs}.
\newblock Springer-Verlag, 1989.

\bibitem{ahslwedekatona78}
R.~Ahlswede and G.~Katona.
\newblock Graphs with maximal number of adjacent pairs of edges.
\newblock {\em Acta Math. Acad. Sci. Hungar.}, 32:97--120, 1978.

\bibitem{bey2006}
C.~Bey.
\newblock Remarks on an edge-isoperimetric problem.
\newblock In {\em General theory of information transfer and combinatorics,
  Lecture Notes in Comput. Sci., 4123}, pages 971--978. Springer, Berlin, 2006.

\bibitem{bezrukov90}
S.~L. Bezrukov.
\newblock An isoperimetric problem for manhattan lattices.
\newblock In {\em Proc. Int. Conf. Sets, Graphs and Numbers, Budapest}, pages
  2--3, 1991.

\bibitem{bezrukovleck09}
S.~L. Bezrukov and U.~Leck.
\newblock A simple proof of the {K}arakhanyan-{R}iordan theorem on the even
  discrete torus.
\newblock {\em SIAM J. Discrete Math.}, 23(3):1416--1421, 2009.

\bibitem{bezrukovserra02}
S.~L. Bezrukov and O.~Serra.
\newblock A local-global principle for vertex-isoperimetric problems.
\newblock {\em Discrete Math.}, 257(2-3):285--309, 2002.
\newblock Kleitman and combinatorics: a celebration (Cambridge, MA, 1999).

\bibitem{bollobas}
B.~Bollob{{\'a}}s.
\newblock {\em Combinatorics}.
\newblock Cambridge University Press, Cambridge, 1986.
\newblock Set systems, hypergraphs, families of vectors and combinatorial
  probability.

\bibitem{BL90}
B.~Bollob{{\'a}}s and I.~Leader.
\newblock Compressions and isoperimetric inequalities.
\newblock {\em J. Combin. Theory Ser. A}, 56(1):47--62, 1991.

\bibitem{BL04}
B.~Bollob{{\'a}}s and I.~Leader.
\newblock Isoperimetric problems for {$r$}-sets.
\newblock {\em Combin. Probab. Comput.}, 13(2):277--279, 2004.

\bibitem{Cam99}
P.~J. Cameron.
\newblock {\em Permutation groups}, volume~45 of {\em London Mathematical
  Society Student Texts}.
\newblock Cambridge University Press, Cambridge, 1999.

\bibitem{keevash}
D.~Christofides, D.~Ellis, and P.~Keevash.
\newblock An approximate isoperimetric inequality for {$r$}-sets.
\newblock {\em Electron. J. Combin.}, 20(4):Paper 15, 12, 2013.

\bibitem{ekr}
P.~Erd{\H{o}}s, C.~Ko, and R.~Rado.
\newblock Intersection theorems for systems of finite sets.
\newblock {\em Quart. J. Math. Oxford Ser. (2)}, 12:313--320, 1961.

\bibitem{franklfuredi81}
P.~Frankl and Z.~F{{\"u}}redi.
\newblock A short proof for a theorem of {H}arper about {H}amming-spheres.
\newblock {\em Discrete Math.}, 34(3):311--313, 1981.

\bibitem{furedigriggs}
Z.~F{{\"u}}redi and J.~R. Griggs.
\newblock Families of finite sets with minimum shadows.
\newblock {\em Combinatorica}, 6(4):355--363, 1986.

\bibitem{esq}
C.~D. Godsil.
\newblock {\em Algebraic combinatorics}.
\newblock Chapman and Hall Mathematics Series. Chapman \& Hall, New York, 1993.

\bibitem{godsil1993}
C.~D. Godsil.
\newblock {\em Algebraic combinatorics}.
\newblock Chapman and Hall Mathematics Series. Chapman \& Hall, New York, 1993.

\bibitem{harper66}
L.~H. Harper.
\newblock Optimal numberings and isoperimetric problems on graphs.
\newblock {\em J. Combinatorial Theory}, 1:385--393, 1966.

\bibitem{Harper04}
L.~H. Harper.
\newblock {\em Global methods for combinatorial isoperimetric problems},
  volume~90 of {\em Cambridge Studies in Advanced Mathematics}.
\newblock Cambridge University Press, Cambridge, 2004.

\bibitem{Hart76}
S.~Hart.
\newblock A note on the edges of the {$n$}-cube.
\newblock {\em Disc. Math.}, 14:157--163, 1976.

\bibitem{Karachanjan82}
V.~M. Karachanjan.
\newblock A discrete isoperimetric problem on multidimensional torus.
\newblock {\em Doklady AN Arm. SSR}, 74(2):61--65, 1982.

\bibitem{Katona68}
G.~Katona.
\newblock A theorem of finite sets.
\newblock In {\em Theory of graphs ({P}roc. {C}olloq., {T}ihany, 1966)}, pages
  187--207. Academic Press, New York, 1968.

\bibitem{Kruskal63}
J.~B. Kruskal.
\newblock The number of simplices in a complex.
\newblock In {\em Mathematical optimization techniques}, pages 251--278. Univ.
  of California Press, Berkeley, Calif., 1963.

\bibitem{leader91}
I.~Leader.
\newblock Discrete isoperimetric inequalities.
\newblock In {\em Probabilistic combinatorics and its applications ({S}an
  {F}rancisco, {CA}, 1991)}, volume~44 of {\em Proc. Sympos. Appl. Math.},
  pages 57--80. Amer. Math. Soc., Providence, RI, 1991.

\bibitem{mors}
M.~M{{\"o}}rs.
\newblock A generalization of a theorem of {K}ruskal.
\newblock {\em Graphs Combin.}, 1(2):167--183, 1985.

\bibitem{Riordan98}
O.~Riordan.
\newblock An ordering on the even discrete torus.
\newblock {\em SIAM J. Discrete Math.}, 11(1):110--127, 1998.

\bibitem{cioaba2014}
J.~K. S.M. Cioab\v~a and W.~Li.
\newblock Disconnecting strongly regular graphs.
\newblock {\em European J. Combinatorics}, 38:1--11, 2014.

\bibitem{cioaba2012}
K.~K. S.M. Cioab\v~a and J.~Koolen.
\newblock On a conjecture of brouwer involving the connectivity of strongly
  regular graphs.
\newblock {\em Journal of Combinatorial Theory, Series A}, 119:904--922, 2012.

\end{thebibliography}

\end{document}